\documentclass[12pt]{amsart}

\usepackage{color}

\usepackage{latexsym,amssymb}
\usepackage{amsmath}
\usepackage[mathscr]{eucal}
\usepackage{color}
\usepackage[abbrev]{amsrefs}
\usepackage[T1]{fontenc}

\usepackage[pagewise]{lineno}
\newcommand*\patchAmsMathEnvironmentForLineno[1]{
  \expandafter\let\csname old#1\expandafter\endcsname\csname #1\endcsname
  \expandafter\let\csname oldend#1\expandafter\endcsname\csname end#1\endcsname
  \renewenvironment{#1}
     {\linenomath\csname old#1\endcsname}
     {\csname oldend#1\endcsname\endlinenomath}}
\newcommand*\patchBothAmsMathEnvironmentsForLineno[1]{
  \patchAmsMathEnvironmentForLineno{#1}
  \patchAmsMathEnvironmentForLineno{#1*}}
\AtBeginDocument{
\patchBothAmsMathEnvironmentsForLineno{equation}
\patchBothAmsMathEnvironmentsForLineno{align}
\patchBothAmsMathEnvironmentsForLineno{flalign}
\patchBothAmsMathEnvironmentsForLineno{alignat}
\patchBothAmsMathEnvironmentsForLineno{gather}
\patchBothAmsMathEnvironmentsForLineno{multline}
}
\newtheorem{Theorem}{\bf Theorem}[section]
\newtheorem{Lemma}{\bf Lemma}[section]
\newtheorem{Proposition}{\bf Proposition}[section]  
\newtheorem{Corollary}{\bf Corollary}[section]
\newtheorem{Remark}{\bf Remark}[section]
\newtheorem{Example}{\bf Example}[section]
\newtheorem{Definition}{\bf Definition}[section]
\newtheorem{Conjecture}{\bf Conjecture}[section]

\newenvironment{theorem}{\begin{Theorem}$\!\!\!$}{\end{Theorem}}
\newenvironment{lemma}{\begin{Lemma}$\!\!\!$}{\end{Lemma}}

\newenvironment{corollary}{\begin{Corollary}$\!\!\!$}{\end{Corollary}}
\newenvironment{remark}{\begin{Remark}$\!\!\!$}{\end{Remark}}

\newenvironment{definition}{\begin{Definition}$\!\!\!$}{\end{Definition}}

\numberwithin{equation}{section}

\usepackage[margin=2.2cm]{geometry}

\newcommand{\RN}{{\bf{R}}^N}

\def\Xint#1{\mathchoice
{\XXint\displaystyle\textstyle{#1}}%
{\XXint\textstyle\scriptstyle{#1}}%
{\XXint\scriptstyle\scriptscriptstyle{#1}}%
{\XXint\scriptscriptstyle\scriptscriptstyle{#1}}%
\!\int}
\def\XXint#1#2#3{{\setbox0=\hbox{$#1{#2#3}{\int}$}
\vcenter{\hbox{$#2#3$}}\kern-.5\wd0}}

\def\dashint{\Xint-}

\makeatletter
\@namedef{subjclassname@2020}{\textup{2020} Mathematics Subject Classification}
\makeatother

\begin{document}
%%%%%%%%%%%%%%%%%%%%%%%%%%%%%%%%%%
%%%%%%%%%%%%%%%%%%%%%%%%%%%%%%%%%%
%%%%%%%%%%%%%%%%%%%%%%%%%%%%%%%%%%
\title[The Hardy parabolic equation]
{Existence and nonexistence of solutions to \\the Hardy parabolic equation}

\author[K. Hisa]{Kotaro Hisa}
\address{Graduate School of Mathematical Sciences, 
The University of Tokyo, 3-8-1 Komaba, Meguro-ku, Tokyo 153-8914, Japan}
\email[Corresponding author]{khisa@ms.u-tokyo.ac.jp}
\thanks{The first author was supported 
by JSPS KAKENHI Grant Number JP19H05599.}

\author[M. Sier\.{z}\k{e}ga]{Miko{\l}aj Sier\.{z}\k{e}ga}
\address{Faculty of Mathematics, Informatics and Mechanics, University of Warsaw, Banacha 2, 02-097 Warsaw, Poland}
\email{m.sierzega@mimuw.edu.pl}
\thanks{The second author was supported by NCN grant 2017/26/D/ST1/00614}



\begin{abstract}
In this paper, we obtain necessary conditions and sufficient conditions on the initial data for the local-in-time solvability of the Cauchy problem
\[
\partial_t u +(-\Delta)^\frac{\theta}{2} u=|x|^{-\gamma} u^p ,\quad  x\in{\bf R}^N, t>0, \qquad
u(0)=\mu  \quad  \mbox{in} \quad {\bf R}^N,\vspace{3pt}
\]
where $N\ge 1$,  $0<\theta\le2$, $p>1$, $\gamma>0$ and $\mu$ is a nonnegative Radon measure on ${\bf R}^N$. 
Using these conditions, we attempt to identify the optimal strength of the singularity of $\mu$ for the existence of solutions to this problem.
\end{abstract}

\maketitle

%%%%%%%%%%%%%%%%%%%%%%%%%%%%%%%%%%%%
%%%%%%%%%%%%%%%%%%%%%%%%%%%%%%%%%%%%
\section{Introduction}
%%%%%%%%%%%%%%%%%%%%%%%%%%%%%%%%%%%%
%%%%%%%%%%%%%%%%%%%%%%%%%%%%%%%%%%%%
Consider nonnegative solutions to the Cauchy problem for the Hardy parabolic equation
\begin{equation}
\label{eq:1.1}
\left\{
\begin{array}{ll}
\partial_t u +(-\Delta)^\frac{\theta}{2} u=|x|^{-\gamma} u^p ,\quad & x\in{\bf R}^N,\,\,\,t>0,\vspace{3pt}\\
u(0)=\mu &   \mbox{in} \quad {\bf R}^N,\vspace{3pt}
\end{array}
\right.
\end{equation}
where $N\ge 1$,  $0<\theta\le2$, $p>1$, $\gamma>0$ and $\mu$ is a nonnegative Radon measure in ${\bf R}^N$. 
Here $(-\Delta)^{\theta/2}$ denotes the fractional power of the Laplace operator $-\Delta$ in $\RN$.
If $\gamma=0$, this equation is the Fujita-type equation.
Throughout this paper  we assume 
\[
0<\gamma<\min\{\theta,N\}.
\]
In the case of $0<\theta<2$, we assume the additional condition
\begin{equation}
\label{eq:1.2}
\gamma<\theta(p-1).
\end{equation}
In the case of $\theta=2$, \eqref{eq:1.2} is not necessary.
In this paper, we attempt to give necessary conditions and sufficient conditions for the local-in-time solvability of the Cauchy problem~\eqref{eq:1.1} and to identify the optimal strength of the singularity of $\mu$ for the existence of local-in-time solutions to \eqref{eq:1.1}.
The potential term $|x|^{-\gamma}$ promotes blow-up of  solutions to \eqref{eq:1.1} at the origin, on the other hand, its effect is hardly noticeable far from the origin. 
For this reason,
far from the origin, the optimal singularity of $\mu$ is expected to be same as that of the Fujita-type equation, 
while at the origin, it is expected to be weaker than that of the Fujita-type equation.
However, to our knowledge, there seem to be no results describing this prediction.

We recall the local-in-time solvability of the Fujita-type equation and the optimal singularity of its initial data.
Let us consider nonnegative solutions to the semilinear parabolic equation
\begin{equation}
\label{eq:fjt}
\partial_t v +(-\Delta)^\frac{\theta}{2} v= v^p ,\quad  x\in{\bf R}^N,\,\,\,t>0, \qquad
v(0)=\nu \quad  \mbox{in} \quad {\bf R}^N,\vspace{3pt}
\end{equation}
where $N\ge 1$,  $0<\theta\le2$, $p>1$ and $\nu$ is a nonnegative Radon measure in ${\bf R}^N$. 
The local-in-time solvability of Cauchy problem~\eqref{eq:fjt} has been studied in many papers (see e.g. \cite{AD, BP, BC, FI, HI01, IKK, IKO, IKS, KY, RS, SL01, SL02, T, W01, W02} and references therein).
Among them, Baras--Pierre~\cite{BP}
obtained  necessary conditions for the local-in-time solvability in the case of $\theta=2$.
Subsequently, the first author of this paper and Ishige~\cite{HI01} obtained a generalization to the case of $0<\theta\le2$.
Precisely, the following results have already been obtained.

\begin{itemize}
\item[(a)] Let $0<T<\infty$. Assume that problem~\eqref{eq:fjt} possesses a nonnegative solution. Then $\nu$ must satisfy the following:
\begin{itemize}
\item If $1<p<p_0$, then ${\sup_{x\in\RN} \nu(B(x,1))<\infty}$;
\item If $p=p_0$, then ${\sup_{x\in\RN} \nu(B(x,\sigma))<c_* \left[\log\left(e+\frac{T^\frac{1}{\theta}}{\sigma}\right)\right]^{-\frac{N}{\theta}}}$ for all $0<\sigma<T^\frac{1}{\theta}$;
\item If $p>p_0$, then ${\sup_{x\in\RN} \nu(B(x,\sigma))<c_* \sigma^{N-\frac{\theta}{p-1}}}$ for all $0<\sigma<T^\frac{1}{\theta}$.
\end{itemize}
Here $p_0:=1+\theta/N$ and $c_*$ is a constant depending only on $N$, $\theta$ and $p$.
\end{itemize}
Then, we can find a  large constant $C_*>0$ with the following property: Set
\begin{equation}
\label{OS}
\Psi(x) := 
\left\{
\begin{array}{ll}
\displaystyle{|x|^{-N}\left[\log\left(e+\frac{1}{|x|}\right)\right]^{-\frac{N}{\theta}-1}} & \mbox{if} \quad p=p_0,\vspace{3pt}\\
\displaystyle{|x|^{-\frac{\theta}{p-1}}},\quad & \mbox{if} \quad p>p_0.\vspace{3pt}\\
\end{array}
\right.
\end{equation}
\begin{itemize}
\item[(b)] Problem~\eqref{eq:fjt} possesses no local-in-time solutions if $\nu$ is a nonnegative measurable function in $\RN$ satisfying
$\nu(x)\ge C_*\Psi(x)$
in a neighborhood of the origin.
\end{itemize}
On the other hand, Kozono--Yamazaki \cite{KY}, Robinson and the second author of this paper \cite{RS} and the first author of this paper and Ishige \cite{HI01} obtained  sufficient conditions for the local-in-time solvability. These results proved that there exists a small constant $c_*>0$ 
such that
if $\nu$ satisfies
$0\le\nu(x)\le c_*\Psi(x)$ in $\RN$,
then problem~\eqref{eq:fjt} possesses a local-in-time solution.
By combining these conditions, we see that $\Psi(x)$ is the optimal singularity of  $\nu$ for the existence of local-in-time solutions to problem~\eqref{eq:fjt}.
We are interested in finding similar necessary conditions and sufficient conditions for the local-in-time solvability of the Cauchy problem~\eqref{eq:1.1} and identifying the optimal singularity of the initial data.

Now let us return to the Cauchy problem~\eqref{eq:1.1}.
This problem has been studied in \cite{AT,Ben, STW, C19, CIT, FT, H, MM, P, Qi, Tayachi, W}.
Among them, in the case of $\theta=2$, Ben Slimene--Tayachi--Weissler \cite{STW} proved that the Cauchy problem~\eqref{eq:1.1} is locally well-posed in $L^q(\RN)$ with a suitable $q\ge1$. Moreover, in the case of $p>1+(2-\gamma)/N$, they proved that if $\mu$ satisfies
\[
0\le\mu(x)\le c|x|^{-\frac{2-\gamma}{p-1}} \quad \mbox{in} \quad \RN
\]
for sufficiently small $c>0$, then \eqref{eq:1.1} possesses a global-in-time solution.
However, there seem to be no results on necessary conditions such as assertion~(a).
Further still, it seems that in the case of $0<\theta <2$ no results covering the sufficient conditions are available. 

In order to state our main results,
we introduce some notation and formulate the definition of solutions to \eqref{eq:1.1}.
For $x\in\RN$ and $r>0$, let $B(x,r):= \{y\in\RN:|x-y|<r\}$ and $|B(x,r)|$ be the volume of $B(x,r)$.
Furthermore, for $f\in L^1_{loc}(\RN)$, we set
\[
\dashint_{B(x,r)} f(y) \,dy := \frac{1}{|B(x,r)|}\int_{B(x,r)} f(y) \, dy.
\]
Let $G=G(x,t)$ be the fundamental solution to 
\begin{equation}
\label{heateq}
\partial_t v + (- \Delta)^\frac{\theta}{2} v =0 \quad \mbox{in} \quad \RN \times (0,\infty),
\end{equation}
where $0<\theta\le2$.

\begin{definition}
\label{Definition:1.1}
Let $u$ be a nonnegative measurable function in $\RN\times(0,T)$, where $0<T\le\infty$.
We say that $u$ is a solution to \eqref{eq:1.1} in $\RN\times[0,T)$ if $u$ satisfies
\begin{equation}
\label{eq:1.3}
\infty>u(x,t) = \int_{\RN} G(x-y,t)\, d\mu(y) + \int_0^t \int_{\RN} G(x-y, t-s) |y|^{-\gamma} u(y,s)^p \, dyds
\end{equation}
for almost all $x\in\RN$ and $0<t<T$.
\end{definition}
In what follows, define 
\[
p_\gamma:=1+\frac{\theta-\gamma}{N}
\]
for $N\ge1$, $0<\theta\le2$ and  $\gamma\ge0$.

Now we are ready to state the main results of this paper. In Theorem~\ref{Theorem:1.1} and \ref{Theorem:1.2} we obtain necessary conditions for the local-in-time solvability of the Cauchy problem~\eqref{eq:1.1}.

\begin{theorem}
\label{Theorem:1.1}
Let $u$ be a solution to \eqref{eq:1.1} in $\RN\times[0,T)$, where $0<T<\infty$.
Then, there exists a constant $C_1>0$ depending only on $N$, $\theta$, $p$ and $\gamma$, such that
\begin{equation}
\label{eq:1.5}
\sup_{z\in\RN} \left(\dashint_{B(z,\sigma)} |x|^\frac{\gamma}{p-1} \, dx\right)^{-1} \mu(B(z,\sigma)) \le C_1 \sigma^{N-\frac{\theta}{p-1}}
\end{equation}
for all $0<\sigma<T^\frac{1}{\theta}$. In particular, in the case of $p=p_\gamma$, there exists a constant $C_1'>0$ depending only on $N$, $\theta$ and $\gamma$, such that
\begin{equation}
\label{eq:1.6}
\mu(B(0,\sigma)) \le C_1'\left[\log\left(e+\frac{T^\frac{1}{\theta}}{\sigma}\right)\right]^{-\frac{N}{\theta-\gamma}}
\end{equation}
for all $0<\sigma<T^{1/\theta}$.
\end{theorem}
Since the effect of the singularity $|x|^{-\gamma}$ is hardly noticeable away from the origin, the solution to \eqref{eq:1.1} is expected to behave like that to \eqref{eq:fjt}.
Theorem~\ref{Theorem:1.2} shows a necessary condition for the local-in-time solvability of the Cauchy problem~\eqref{eq:1.1} far from the origin in the case of $p=p_0$ (compare with (a) as above).

\begin{theorem}
\label{Theorem:1.2}
Let  $z\in\RN$ and $p=p_0$. Let $u$ be a solution to \eqref{eq:1.1} in $\RN\times[0,T)$, where $0<T<\infty$. Assume that $|z|>T^{1/\theta}$.
Then, there exists a constant $C_2>0$ depending only on $N$, $\theta$ and $\gamma$, such that
\begin{equation*}
%\label{eq:1.7}
\mu(B(z,\sigma)) \le C_2 |z|^\frac{\gamma}{p-1} \left[\log\left(e+\frac{T^\frac{1}{\theta}}{\sigma}\right)\right]^{-\frac{N}{\theta}}
\end{equation*}
for all $0<\sigma<T^\frac{1}{\theta}$.
\end{theorem}
From Theorems~\ref{Theorem:1.1} and \ref{Theorem:1.2}, we have the following remarks. 
\begin{remark}
There exists a constant $C_*>0$ with the following property:
\begin{itemize}
\item[\rm{(i)}]  Problem~\eqref{eq:1.1} possesses no local-in-time solution if $\mu$ is a nonnegative measurable function in $\RN$ satisfying
\begin{equation*}
\mu(x)\ge
\left\{
\begin{array}{ll}
\displaystyle{C_* |x|^{-N}\left[\log\left(e+\frac{1}{|x|}\right)\right]^{-\frac{N}{\theta-\gamma}-1}} & \mbox{if} \quad p=p_\gamma,\vspace{3pt}\\
\displaystyle{C_* |x|^{-\frac{\theta-\gamma}{p-1}}},\quad & \mbox{if} \quad p>p_\gamma,\vspace{3pt}\\
\end{array}
\right.
\end{equation*}
in the neighborhood of the origin.

\item[\rm{(ii)}] Let $z\in\RN\setminus\{0\}$. Problem~\eqref{eq:1.1} possesses no local-in-time solution if $\mu$ is a nonnegative measurable function in $\RN$ satisfying
\begin{equation*}
\mu(x)\ge  C_* |z|^\frac{\gamma}{p-1} \Psi(x)
\end{equation*}
in the neighborhood of $z$.
\end{itemize}
\end{remark}
\begin{remark}
Let $\mu = \delta_z$ in $\RN$, where $\delta_z$ is the Dirac measure concentrated at $z\in\RN$. Then, the following holds:
\begin{itemize}
\item If $p_\gamma\le p $ and $z=0$, problem~\eqref{eq:1.1} possesses no local-in-time solution;

\item If $p_0\le p$ and $z\in\RN$, problem~\eqref{eq:1.1} possesses no local-in-time solution.
\end{itemize}
\end{remark}
The proof of \eqref{eq:1.5} is based on \cite[Theorem 1.1]{HIT} and \cite[Proposition 1]{KK}.
In this proof, we consider the weak solution to \eqref{eq:1.1} and give an upper bound for $\mu$ by substituting a suitable test function. 
On the other hand, the proofs of \eqref{eq:1.6} and Theorem~\ref{Theorem:1.2} are based on \cite[Lemma~3.2]{HI01}, which proved a necessary condition in the case of $\gamma=0$. Let $u$ be a solution (in the sense of Definition~\ref{Definition:1.1}) to \eqref{eq:1.1} in $\RN\times[0,T)$ and $z\in\RN$.
Following \cite{HI01}, we employ an iteration argument to get a lower estimate related to
\[
\int_{\RN} G(x,t) u(x+z,t) \, dx,
\]
and we prove \eqref{eq:1.6} and Theorem~\ref{Theorem:1.2}.
In particular, in the proof of Theorem~\ref{Theorem:1.2}, 
in order to apply an argument used previously for the Fujita-type  equation \cite{HI01},
we have to estimate the potential term $|x|^{-\gamma}$ properly, and  assumption $T^{1/\theta}<|z|$ plays an important role in it.
Moreover, it is important to estimate the integral 
\[
\int_{\RN} G(y,t) |y+z|^\frac{\gamma}{p-1} \, dy
\]
from above for $t>0$, and the assumption \eqref{eq:1.2} guarantees that this value is finite (see \eqref{eq:2.2} below).

We give a sufficient condition for the local-in-time solvability of problem~\eqref{eq:1.1}.

\begin{theorem}
\label{Theorem:1.3}
Let $p>p_\gamma$, $\alpha>1$, $r>1$ and $T>0$. Assume that $r>1$ satisfies
\begin{equation}
\label{eq:1.8}
r\in \left(\frac{N(p-1)}{\theta-\gamma}-\epsilon, \frac{N(p-1)}{\theta-\gamma}\right)
\end{equation} 
for sufficiently small $\epsilon>0$.
 Then there exists a constant $C_3>0$, depending only on $N$, $\theta$, $p$, $\gamma$ and $r$, such that, if $\mu$ is a nonnegative measurable function in $\RN$ satisfying 
\begin{equation}
\label{eq:1.9}
\sup_{z\in\RN} \left(\dashint_{B(z,\sigma)} \mu(y)^r \, dy \right)^\frac{1}{r} \le C_3 \sigma^{-\frac{\theta-\gamma}{p-1}}
\end{equation}
for all $0<\sigma<T^{1/\theta}$, then problem~\eqref{eq:1.1} possesses a solution in $\RN\times[0,T)$.
\end{theorem}
As a corollary of Theorem~\ref{Theorem:1.3}, we have
\begin{corollary}
\label{Corollary:1.1}
Let $p>p_\gamma$.
There exists a constant $c_*>0$ such that
if $\mu$ satisfies
\begin{equation}
\label{eq:1.12}
\displaystyle{0\le\mu(x)\le c_* |x|^{-\frac{\theta-\gamma}{p-1}}} \quad \mbox{in} \quad \RN,
\end{equation}
then problem~\eqref{eq:1.1} possesses a local-in-time solution.
\end{corollary}
By Corollary~\ref{Corollary:1.1}, we see that the singularity of $\mu$ as in \eqref{eq:1.12} is the optimal one at the origin in the case of $p>p_\gamma$.
However, in other cases, the optimal singularity has not been obtained.

The rest of this paper is organized as follows. 
In Section~2 we collect some properties of the fundamental solution $G$ and prepare some preliminary lemmas. 
In Section~3 we prove Theorem~\ref{Theorem:1.1}.
In Section~4 we prove Theorem~\ref{Theorem:1.2}.
In Section~5 we prove Theorem~\ref{Theorem:1.3}.

%%%%%%%%%%%%%%%%%%%%%%%%%%%%%%%%%%%%
%%%%%%%%%%%%%%%%%%%%%%%%%%%%%%%%%%%%
\section{Preliminaries}
%%%%%%%%%%%%%%%%%%%%%%%%%%%%%%%%%%%%
%%%%%%%%%%%%%%%%%%%%%%%%%%%%%%%%%%%%
In this section, we collect some properties of the fundamental solution $G$ to \eqref{heateq} 
and prepare preliminary lemmas. 
In what follows 
the letter $C$ denotes a generic positive constant depending only on $N$, $\theta$, $p$ and $\gamma$.
 
Let $N\ge1$ and $0<\theta\le2$.
The fundamental solution $G$ to \eqref{heateq} is a positive and smooth function in $\RN\times(0,\infty)$ and has the following properties: 
\begin{eqnarray}
\label{eq:2.1}
 & & G(x,t)=t^{-\frac{N}{\theta}}G\left(t^{-\frac{1}{\theta}}x,1\right),\\
\label{eq:2.2}
 & & C^{-1}(1+|x|)^{-N-\theta}\le G(x,1)\le C(1+|x|)^{-N-\theta}\quad\mbox{ if $0<\theta<2$},\\
\label{eq:2.3}
 & &  \mbox{$G(\cdot,1)$ is radially symmetric and $G(x,1)\le G(y,1)$ if $|x|\ge |y|$},\\
\label{eq:2.4}
 & & G(x,t)=\int_{{\bf R}^N}G(x-y,t-s)G(y,s)dy,\\
\label{eq:2.5}
 & & \int_{{\bf R}^N}G(x,t)\,dx=1,
\end{eqnarray}
for all $x$, $y\in{\bf R}^N$ and $0<s<t$.
For any $\phi\in L^1_{loc}(\RN)$, we identify $\phi$ with the Radon measure $\phi\,dx$.
For any Radon measure $\mu$ in $\RN$, we define
\[
[S(t)\mu](x) := \int_{\RN} G(x-y,t) \,d\mu(y), \quad x\in\RN, t>0.
\]
Furthermore, we have the following lemmas. 

\begin{lemma}
\label{Lemma:2.1}
There exists a constant $C>0$ such that
\begin{equation*}
%\label{eq:2.6}
\|S(t)\mu\|_{L^\infty(\RN)} \le C t^{-\frac{N}{\theta}} \sup_{z\in\RN} \mu({B(x,t^\frac{1}{\theta})})
\end{equation*}
for any Radon measure $\mu$ in $\RN$ and $t>0$.
\end{lemma}
\begin{proof}
This lemma was proved in \cite[Lemma~2.1]{HI01}.
\end{proof}

\begin{lemma}
\label{Lemma:2.2}
Let $\mu$ be a nonnegative Radon measure in ${\bf R}^N$ and $0<T\le\infty$. Assume that there exists a supersolution $v$ to \eqref{eq:1.1} in $\RN\times[0,T)$. Then, there exists a solution to \eqref{eq:1.1} in $\RN\times[0,T)$.
\end{lemma}
\begin{proof}
Set $u_1:=S(t)\mu$. Define $u_k \, (k=2,3,\cdots)$ inductively by
\begin{equation}
\label{eq:2.7}
u_k(t):=S(t)\mu+\int_0^tS(t-s)|\cdot|^{-\gamma}u_{k-1}(s)^p.
\end{equation}
Let $v$ be a supersolution to \eqref{eq:1.1} in $\RN\times[0,T)$, where $0<T\le\infty$. Then, it follows inductively that
\[
0\le u_1(x,t) \le u_2(x,t)\le \cdots \le u_k(x,t) \le \cdots \le v(x,t) <\infty
\]
for almost all $x\in\RN$ and $t\in(0,T)$. This implies that
\[
u(x,t):=\lim_{k\to\infty} u_k(x,t)
\]
is well-defined for almost all $x\in\RN$ and $t\in(0,T)$.
Furthermore, by \eqref{eq:2.7}, we see that $u$ satisfies \eqref{eq:1.3} for almost all $x\in\RN$ and $t\in(0,T)$.  
Thus, Lemma~\ref{Lemma:2.2} follows.
\end{proof}
%%%%%%%%%%%%%%%%%%%%%%%%%%%%%%%%%%%%
%%%%%%%%%%%%%%%%%%%%%%%%%%%%%%%%%%%%
\section{Proof of Theorem~\ref{Theorem:1.1}}
%%%%%%%%%%%%%%%%%%%%%%%%%%%%%%%%%%%%
%%%%%%%%%%%%%%%%%%%%%%%%%%%%%%%%%%%%
In this section, we prove Theorem~\ref{Theorem:1.1}.
First, we prove \eqref{eq:1.5}. 
The proof is based on the argument in \cite[Theorem 1.1]{HIT} and \cite[Proposition 1]{KK}.
Therefore, we note the following remark.
\begin{remark}
Let $0<T<\infty$.
If $u$ satisfies \eqref{eq:1.3}, then $u$ also satisfies
\begin{equation}
\label{eq:1.4}
\int_0^T \int_{\RN} (u(-\varphi_t+ (-\Delta)^\frac{\theta}{2} \varphi) - |x|^{-\gamma} u^p \varphi) \,dxdt = \int_{\RN} \varphi(0) \, d\mu
\end{equation}
for all $\varphi\in C^\infty_0(\RN \times [0,T])$ with $\varphi(T)=0$.
\end{remark}
The proof of \eqref{eq:1.5} relies on substituting a suitable test function into \eqref{eq:1.4}.
\begin{proof}[Proof of \eqref{eq:1.5}]
Let $u$ be a solution to \eqref{eq:1.1} in $\RN\times[0,T)$, where $0<T<\infty$. For $z\in\RN$ and $\sigma\in(0,T^{1/\theta})$, let $\zeta:\RN\times[0,\infty)\to[0,1]$ be a $C^\infty_0$-function which satisfies
\begin{equation*}
\zeta(x,t)= 
\left\{
\begin{array}{ll}
1 & \quad\mbox{in} \quad \overline{B(z,{2^{-{1/\theta}}\sigma})}\times[0,2^{-1}\sigma^\theta],\vspace{3pt}\\
0 & \quad\mbox{outside} \quad B(z,\sigma)\times[0,\sigma^\theta).\vspace{3pt}
\end{array}
\right.
\end{equation*}
By substituting $\varphi(x,t)=\zeta(x,t)^s$ as a test function in \eqref{eq:1.4}, where $s$ is an integer and satisfies $s>p/(p-1)$, we get
\begin{equation}
\label{eq:3.1}
\begin{split}
&-s\int_0^{\sigma^\theta} \int_{B(z,\sigma)} u \zeta_t \zeta^{s-1} \, dxdt
+\int_0^{\sigma^\theta} \int_{\RN} u (-\Delta)^\frac{\theta}{2} \zeta^s \, dxdt\\
& = \int_0^{\sigma^\theta}\int_{B(z,\sigma)} |x|^{-\gamma} u^p \zeta^s \, dxdt + 
\int_{B(z,\sigma)} \zeta(0)^s \,d\mu.
\end{split}
\end{equation}
It follows from \eqref{eq:3.1} and the Young inequality that
\begin{equation*}
%\label{eq:3.2}
\begin{split}
&\int_0^{\sigma^\theta}\int_{B(z,\sigma)} |x|^{-\gamma} u^p \zeta^s \, dxdt + 
\int_{B(z,\sigma)} \zeta(0)^s \,d\mu\\
& = -s\int_0^{\sigma^\theta} \int_{B(z,\sigma)} u \zeta_t \zeta^{s-1} \, dxdt
+\int_0^{\sigma^\theta} \int_{\RN} u (-\Delta)^\frac{\theta}{2} \zeta^s \, dxdt\\
& \le -s\int_0^{\sigma^\theta} \int_{B(z,\sigma)} u \zeta_t \zeta^{s-1} \, dxdt
+s \int_0^{\sigma^\theta} \int_{B(z,\sigma)} u \zeta^{s-1} (-\Delta)^\frac{\theta}{2} \zeta \, dxdt\\
& \le  C \int_0^{\sigma^\theta} \int_{B(z,\sigma)} u |\zeta_t| \zeta^{s-1} \, dxdt
+ C \int_0^{\sigma^\theta} \int_{B(z,\sigma)} u \zeta^{s-1} |(-\Delta)^\frac{\theta}{2} \zeta| \, dxdt\\
& \le \int_0^{\sigma^\theta} \int_{B(z,\sigma)} |x|^{-\gamma} u^p\zeta^s \, dxdt
+ C \int_0^{\sigma^\theta} \int_{B(z,\sigma)} |x|^\frac{\gamma}{p-1} |\zeta_t|^\frac{p}{p-1} \zeta^{s-\frac{p}{p-1}} \, dxdt\\
& \quad \quad \quad \quad \quad \quad \quad \quad + C \int_0^{\sigma^\theta} \int_{B(z,\sigma)} |x|^\frac{\gamma}{p-1} |(-\Delta)^\frac{\theta}{2}\zeta|^\frac{p}{p-1} \zeta^{s-\frac{p}{p-1}} \, dxdt.
\end{split}
\end{equation*}
Here, we also used the inequality $(-\Delta)^{\theta/2}\zeta^s \le s \zeta^{s-1} (-\Delta)^{\theta/2} \zeta$ (see \cite{Ju} for details). Since $s>p/(p-1)$,  we have
\begin{equation}
\label{eq:3.2}
\begin{split}
&\int_{B(z,\sigma)} \zeta(0)^s \,d\mu\\
& \le C \int_0^{\sigma^\theta} \int_{B(z,\sigma)} |x|^\frac{\gamma}{p-1} |\zeta_t|^\frac{p}{p-1}  \, dxdt
+ C \int_0^{\sigma^\theta} \int_{B(z,\sigma)} |x|^\frac{\gamma}{p-1} |(-\Delta)^\frac{\theta}{2}\zeta|^\frac{p}{p-1}  \, dxdt.
\end{split}
\end{equation}
Now, we choose in \eqref{eq:3.2} the function $\zeta(x,t) = \psi(\sigma^{-\theta}t)\xi(\sigma^{-1}x)$, where
$\psi:[0,\infty)\to [0,1]$ is a smooth function which satisfies
\[
\psi(t) = 1 \quad \mbox{on} \quad [0,2^{-1}], \qquad \psi(t) = 0 \quad \mbox{outside} \quad [0,1)
\]
and $\xi:\RN \to [0,1]$ is a smooth function which satisfies 
\[
\xi(x) = 1 \quad \mbox{in} \quad \overline{B(z,2^{-\frac{1}{\theta}})}, \qquad 
\xi(x) = 0 \quad \mbox{in} \quad B(z,1).
\]
Since the functions $\psi$ and $\xi$ can be chosen such that
\begin{equation}
\label{eq:3.3}
|\partial_t \psi(\sigma^{-\theta} t)| \le C \sigma^{-\theta} \quad \mbox{and} \quad 
|(-\Delta)^\frac{\theta}{2} \xi(\sigma^{-1}x)| \le C\sigma^{-\theta}, 
\end{equation}
\eqref{eq:3.2} with \eqref{eq:3.3} yields
\[
\mu(B(z,2^{-\frac{1}{\theta}}\sigma)) \le C\sigma^{-\frac{\theta }{p-1}} \int_{B(z,\sigma)} |x|^{\frac{\gamma}{p-1}}\,dx
\]
for all $z\in\RN$ and $0<\sigma<T^\frac{1}{\theta}$.
Then, we can find a positive constant $m$ depending only $N$, $\theta$, $p$ and $\gamma$, such that
\begin{equation*}
\begin{split}
&\sup_{z\in\RN}\left(\dashint_{B(z,\sigma)} |x|^{\frac{\gamma}{p-1}}\,dx\right)^{-1} \mu(B(z,\sigma))\\
&\le m \sup_{z\in\RN}\left(\dashint_{B(z,\sigma)} |x|^{\frac{\gamma}{p-1}}\,dx\right)^{-1} \mu(B(z,2^{-\frac{1}{\theta}}\sigma)) \le Cm \sigma^{N-\frac{\theta }{p-1}}
\end{split}
\end{equation*}
for all $0<\sigma<T^\frac{1}{\theta}$. Therefore, we obtain the desired estimate and the proof of \eqref{eq:1.5} is complete.
\end{proof}
In order to prove \eqref{eq:1.6} and complete the proof of Theorem~\ref{Theorem:1.1}, we prepare the following lemma, which has been obtained in \cite[Lemma~3.2]{HI01}.

\begin{lemma}
\label{Lemma:3.1}
Let $u$ be a solution to \eqref{eq:1.1} in $\RN\times[0,T)$, where $0<T<\infty$. Let $z\in\RN$ and $\rho>0$ with $(2\rho)^\theta<T$. Then, there exists a constant $c_*>0$ depending only on $N$ such that 
\begin{equation}
\label{eq:3.4}
u(x+z,(2\rho)^\theta)\ge c_* G(x,\rho^\theta)\mu(B(z,\rho)) 
\end{equation}
for almost all $x\in\RN$.
\end{lemma}
Now we are ready to prove \eqref{eq:1.6}. The proof is based on the argument in \cite[Lemma~3.3]{HI01}.
\begin{proof}[Proof of \eqref{eq:1.6}]
%{\bf Proof of \eqref{eq:1.6}.} 
We assume  $p=p_\gamma$.
Let $\nu>0$ be a sufficiently small constant and $\rho$ be such that
\[
0<\rho<(\nu T)^\frac{1}{\theta}.
\]
Set $v(x,t) := u(x, t+(2\rho)^\theta)$ for almost all $x\in\RN$ and $t \in (0, T-(2\rho)^\theta)$. Since $u$ is a solution to \eqref{eq:1.1} in $\RN\times[0,T)$, it follows from \eqref{eq:1.3} that for $0<\tau<T-(2\rho)^\theta$,
\begin{equation}
\label{eq:3.5}
v(x,t) = \int_{\RN} G(x-y,t-\tau)v(y,\tau) \, dy + \int_\tau^t \int_{\RN} G(x-y, t-s) |y|^{-\gamma} v(y,s)^p \, dyds
\end{equation}
holds for almost all $x\in\RN$, $t\in (\tau, T-(2\rho)^\theta)$. 
In the case of $0<\theta<2$, by \eqref{eq:2.1} and \eqref{eq:2.2} we have
\[
G(x-t,\tau) \ge C\tau^{-\frac{N}{\theta}} \left(1+\frac{|x|+|y|}{\tau^\frac{1}{\theta}}\right)^{-N-\theta} \ge C\tau^{-\frac{N}{\theta}} \left(2+\frac{|y|}{\tau^\frac{1}{\theta}}\right)^{-N-\theta} \ge CG(y,\tau)
\]
for all $x\in\RN$ with $|x|<\tau^\frac{1}{\theta}$, $y\in\RN$ and $\tau>0$. 
This time \eqref{eq:3.5} with $t=2\tau$ gives 
\begin{equation}
\label{eq:3.6}
\int_{\RN} G(y,\tau) v(y,\tau) \, dy < \infty
\end{equation}
for almost all $\tau\in(0,[T-(2\rho)^\theta]/2)$.
On the other hand, in the case of $\theta=2$, we have
\[
G(x-y, 2\tau) \ge (8\pi t)^{-\frac{N}{2}} \exp\left(-\frac{2|x|^2+2|y|^2}{8\tau}\right) \ge C (4\pi t)^{-\frac{N}{2}} \exp\left(-\frac{|y|}{4\tau}\right) = CG(y,\tau)
\]
for all $x\in\RN$ with $|x|<\tau^\frac{1}{\theta}$, $y\in\RN$ and $\tau>0$. 
Then \eqref{eq:3.5} with $t=3\tau$ yields 
\begin{equation}
\label{eq:3.7}
\int_{\RN} G(y,\tau) v(y,\tau) \, dy < \infty
\end{equation}
for almost all $\tau\in(0,[T-(2\rho)^\theta]/3)$.
Furthermore, by \eqref{eq:2.4}, \eqref{eq:3.4} and \eqref{eq:3.5} with $\tau=0$ we have
\begin{equation}
\label{eq:3.8}
\begin{split}
&v(x,t)-\int_0^t \int_{\RN} G(x-y, t-s) |y|^{-\gamma} v(y,s)^p \, dyds\\
&\ge c_*\mu(B(0,\rho)) \int_{\RN}G(x-y,t)G(y,\rho^\theta)\,dy\\
&= c_* \mu(B(0,\rho)) G(x,t+\rho^\theta)
\end{split}
\end{equation}
for almost all $x\in\RN$ and $0<t<T-(2\rho)^\theta$, where $c_*$ is the constant in Lemma~\ref{Lemma:3.1}.
Set 
\[
w(t):= \int_{\RN} G(x,t) v(x,t) \,dx.
\]
By \eqref{eq:3.6} and \eqref{eq:3.7}, we see that $w(t)<\infty$ for almost all $t\in(0,[T-(2\rho)^\theta]/3)$. Then, it follows from \eqref{eq:2.4} and \eqref{eq:3.8} that
\begin{equation}
\label{eq:3.9}
\begin{split}
\infty>w(t)
&\ge c_*\mu(B(0,\rho))\int_{\RN} G(x,t+\rho^\theta)G(x,t) \, dx\\
&+ \int_{\RN} \int_0^t \int_{\RN} G(x-y,t-s)G(x,t)|y|^{-\gamma}v(y,s)^p \, dydsdx\\
&\ge c_* \mu(B(0,\rho)) G(0,2t+\rho^\theta) + \int_{\rho^\theta}^t \int_{\RN} G(y,2t-s)|y|^{-\gamma}v(y,s)^p \, dyds
\end{split}
\end{equation}
for almost all $\rho^\theta<t<[T-(2\rho)^\theta]/3$. Now it follows from \eqref{eq:2.1} and \eqref{eq:2.3} that
\begin{equation}
\label{eq:3.10}
\begin{split}
G(y,2t-s) 
&= (2t-s)^{-\frac{N}{\theta}} G\left(\frac{y}{(2t-s)^\frac{1}{\theta}},1\right)\\
&\ge \left(\frac{s}{2t}\right)^{\frac{N}{\theta}}s^{-\frac{N}{\theta}} G\left(\frac{y}{s^\frac{1}{\theta}},1\right)
= \left(\frac{s}{2t}\right)^{\frac{N}{\theta}} G(y.s)\\
\end{split}
\end{equation}
for $y\in\RN$ and $0<s<t$. By \eqref{eq:2.1}, \eqref{eq:3.9} and \eqref{eq:3.10}, we obtain
\begin{equation}
\label{eq:3.11}
\begin{split}
\infty>w(t)
&\ge c_* \mu(B(0,\rho)) G(0,2t+\rho^\theta) + \int_{\rho^\theta}^t \left(\frac{s}{2t}\right)^{\frac{N}{\theta}} \int_{\RN} G(y,s)|y|^{-\gamma}v(y,s)^p \, dyds\\
&\ge C\mu(B(0,\rho))t^{-\frac{N}{\theta}} + \int_{\rho^\theta}^t \left(\frac{s}{2t}\right)^{\frac{N}{\theta}} \int_{\RN} G(y,s)|y|^{-\gamma}v(y,s)^p \, dyds\\
\end{split}
\end{equation}
for almost all $\rho^\theta<t<[T-(2\rho)^\theta]/3$.
By the H\"{o}lder inequality, we have
\begin{equation}
\label{eq:3.12}
\int_{\RN}G(y,s)|y|^{-\gamma}v(y,s)^p \, dy \ge \left(\int_{\RN} G(y,s) |y|^\frac{\gamma}{p-1} \, dy\right)^{-(p-1)} w(s)^p.
\end{equation}
In the case of $\theta=2$, by direct computations we obtain
\begin{equation}
\label{eq:3.13}
\int_{\RN} G(y,s) |y|^\frac{\gamma}{p-1} \, dy \le Cs^\frac{\gamma}{2(p-1)}.
\end{equation}
On the other hand, in the case of $0<\theta<2$, by virtue of assumption~\eqref{eq:1.2} we see that
\begin{equation}
\label{eq:3.14}
\begin{split}
\int_{\RN} G(y,s) |y|^\frac{\gamma}{p-1} \, dy 
&\le Cs^{-\frac{N}{\theta}} \int_{\RN} \left(1+\frac{|y|}{s^\frac{1}{\theta}}\right)^{-N-\theta} |y|^\frac{\gamma}{p-1}\,dy\\
&= Cs^\frac{\gamma}{\theta(p-1)} \int_{\RN} (1+|y|)^{-N-\theta} |y|^\frac{\gamma}{p-1} \, dy\\
&\le Cs^\frac{\gamma}{\theta(p-1)}.
\end{split}
\end{equation}
By \eqref{eq:3.11}, \eqref{eq:3.12}, \eqref{eq:3.13} and \eqref{eq:3.14}, we get
\begin{equation}
\label{eq:3.15}
\infty>w(t)\ge c_1\mu(B(0,\rho))t^{-\frac{N}{\theta}} + c_2t^{-\frac{N}{\theta}}\int_{\rho^\theta}^t s^{\frac{N}{\theta}-\frac{\gamma}{\theta}} w(s)^p\,ds
\end{equation}
for almost all  $\rho^\theta<t<[T-(2\rho)^\theta]/3$, where $c_1>0$ and $c_2>0$ are constants depending only on $N$, $\theta$, $p$ and $\gamma$.

For $k=1,2,\cdots,$ we define the sequence $\{a_k\}$ inductively as
\begin{equation}
\label{eq:3.16}
a_1 := c_1, \qquad a_{k+1} := c_2 a_k^p \frac{p-1}{p^k-1} \quad (k=1,2,\cdots).
\end{equation}
Furthermore, set
\begin{equation}
\label{eq:3.17}
f_k(t) := a_k\mu(B(0,\rho))^{p^{k-1}} t^{-\frac{N}{\theta}} \left(\log\frac{t}{\rho^\theta}\right)^{\frac{p^{k-1}-1}{p-1}}, \quad k=1,2,\cdots.
\end{equation}
We claim that
\begin{equation}
\label{eq:3.18}
w(t)\ge f_k(t), \quad k=1,2,\cdots,
\end{equation}
for almost all $\rho^\theta<t<[T-(2\rho)^\theta]/3$.
By \eqref{eq:3.15}, we see that \eqref{eq:3.18} holds for $k=1$.
We assume that \eqref{eq:3.18} holds with some $k\in\{1,2,\cdots\}$.
Then, due to \eqref{eq:3.15}, we infer  that
\begin{equation*}
\begin{split}
w(t)
&\ge c_2 t^{-\frac{N}{\theta}}\int_{\rho^\theta}^t s^{\frac{N}{\theta}-\frac{\gamma}{\theta}}f_k(s)^p \, ds\\
&= c_2 t^{-\frac{N}{\theta}}\int_{\rho^\theta}^t s^{\frac{N}{\theta}-\frac{\gamma}{\theta}}\left[a_k\mu(B(0,\rho))^{p^{k-1}} s^{-\frac{N}{\theta}} \left(\log\frac{s}{\rho^\theta}\right)^{\frac{p^{k-1}-1}{p-1}}\right]^p \, ds\\
&= c_2 a_k^p \mu(B(0,\rho))^{p^k} t^{-\frac{N}{\theta}} \int_{\rho^\theta}^t s^{-1}\left(\log\frac{s}{\rho^\theta}\right)^{\frac{p^{k}-p}{p-1}} \, ds \\
&= c_2 a_k^p \frac{p-1}{p^k-1} \mu(B(0,\rho))^{p^k} t^{-\frac{N}{\theta}} \left(\log\frac{t}{\rho^\theta}\right)^{\frac{p^{k}-1}{p-1}}\\
&= a_{k+1}  \mu(B(0,\rho))^{p^k}  t^{-\frac{N}{\theta}} \left(\log\frac{t}{\rho^\theta}\right)^{\frac{p^{k}-1}{p-1}} = f_{k+1}(t)
\end{split}
\end{equation*}
for almost all $\rho^\theta<t<[T-(2\rho)^\theta]/3$. Therefore, we conclude that that \eqref{eq:3.18} holds for all $k=1,2,\cdots$.

Next, we claim that there exists a constant $\beta>0$ such that
\begin{equation}
\label{eq:3.19}
a_k\ge \beta^{p^k}, \quad k=1,2,\cdots.
\end{equation}
Set $b_k:=-p^{-k} \log_{a_k}$. We prove that there exists a constant $C>0$ such that $b_k\le C$. By \eqref{eq:3.16}, we see that
\[
-\log{a_{k+1}}= -p\log{a_k} + \log \left[c_2\frac{p-1}{p^k-1}\right].
\]
This implies that
\begin{equation}
\label{eq:3.20}
b_{k+1}- b_{k} = p^{-k-1} \log \left[c_2\frac{p-1}{p^k-1}\right] \le C p^{-k-1}(k+1), \quad k=1,2,\cdots
\end{equation}
for some constant $C>0$. By \eqref{eq:3.20} we see that
\[
b_{k+1}= b_1 + \sum_{j=1}^{k}(b_{j+1}-b_j) \le b_1 + C \sum_{j=1}^{k}p^{-j-1}(j+1) \le C
\]
for $k=1,2,\cdots$. This implies  \eqref{eq:3.19}. Taking a sufficiently small $\nu$ if necessary, by \eqref{eq:3.17}, \eqref{eq:3.18} and \eqref{eq:3.19} we see that
\begin{equation*}
\begin{split}
\infty>w(t)\ge f_{k+1}(t) 
&\le \left[\beta^p \mu(B(0,\rho)) \left(\log\frac{t}{\rho^\theta}\right)^\frac{1}{p-1} \right]^{p^k} t^{-\frac{N}{\theta}}\left( \log \frac{t}{\rho^\theta}\right)^{-\frac{1}{p-1}}\\
&\le \left[\beta^p \mu(B(0,\rho)) \left(\log\frac{T}{5\rho^\theta}\right)^\frac{1}{p-1} \right]^{p^k} t^{-\frac{N}{\theta}}\left( \log \frac{t}{\rho^\theta}\right)^{-\frac{1}{p-1}}, \quad k=1,2,\cdots
\end{split}
\end{equation*}
for almost all $T/5<t<T/4$. Then it follows that
\[
\beta^p\mu(B(0,\rho))\left[\log\frac{T}{5\rho^\theta}\right]^\frac{1}{p-1} \le1,
\]
which in turn implies that 
\begin{equation}
\label{eq:3.21}
\mu(B(0,\rho)) \le C \left[\log\frac{T}{5\rho^\theta}\right]^{-\frac{1}{p-1}} \le C \left[\log\frac{T}{\rho^\theta}\right]^{-\frac{1}{p-1}} \le C \left[\log\left(e+\frac{T}{\rho^\theta}\right)\right]^{-\frac{1}{p-1}}
\end{equation}
for $0<\rho<(\nu T)^{1/\theta}$. 
By \eqref{eq:1.5} there exists a constant $C_*>0$ such that 
\begin{equation*}
%\label{eq:3.22}
\mu(B(0,\rho)) \le C_* 
\end{equation*}
for $(\nu T)^{1/\theta}\le\rho<T^{1/\theta}$. 
Since $p=p_\gamma$, we see that
\begin{equation}
\label{eq:3.22}
\left[\log\left(e+\frac{T}{\rho^\theta}\right)\right]^{-\frac{1}{p-1}}
\ge \left[\log\left(e+\nu^{-\frac{1}{\theta}}\right)\right]^{-\frac{1}{p-1}}
= C
\ge \frac{C}{C_*} \mu(B(0,\rho))
\end{equation}
for $(\nu T)^{1/\theta}\le\rho<T^{1/\theta}$. 
Combining \eqref{eq:3.21} and \eqref{eq:3.22}, we obtain
\[
\mu(B(0,\rho)) \le C \left[\log\left(e+\frac{T}{\sigma^\theta}\right)\right]^{-\frac{1}{p-1}}
\]
for all $0<\sigma<T^{1/\theta}$.
Therefore, we obtain the desired result and the proof of Theorem~\ref{Theorem:1.1} is complete.
\end{proof}
%%%%%%%%%%%%%%%%%%%%%%%%%%%%%%%%%%%%
%%%%%%%%%%%%%%%%%%%%%%%%%%%%%%%%%%%%
\section{Proof of Theorem~\ref{Theorem:1.2}}
%%%%%%%%%%%%%%%%%%%%%%%%%%%%%%%%%%%%
%%%%%%%%%%%%%%%%%%%%%%%%%%%%%%%%%%%%
In this section, we prove Theorem~\ref{Theorem:1.2}.
The proof is based on  the argument found in \cite[Lemma~3.3]{HI01}. 
The following lemma is the key to the proof.
\begin{lemma}
\label{Lemma:4.1}
Assume the same conditions as in Theorem~{\rm\ref{Theorem:1.2}}. Then there exists a constant $C>0$ depending only on $N$, $\theta$ and $\gamma$ such that
\begin{equation}
\label{eq:4.1}
\int_{\RN} G(y,s) |y+z|^\frac{\gamma}{p-1}\, dy \le C\rho^{-N} \int_{B(z,\rho)} |y|^\frac{\gamma}{p-1} \, dy
\end{equation}
for almost all $\rho^\theta<s<T/3$.
\end{lemma}
\begin{proof}
The proof is divided  into two steps.\\
\underline{\bf 1st step.}
We prove that
\[
\int_{\RN} G(y,s) |y+z|^\frac{\gamma}{p-1}\, dy  \le C \int_{B(0,s^\frac{1}{\theta})} G(y,s) |y+z|^\frac{\gamma}{p-1} \, dy
\]
for almost all $\rho^\theta<s<T/3$.
For this purpose, it is sufficient to show that
\begin{equation}
\label{eq:4.2}
\int_{B(0,s^\frac{1}{\theta})^c} G(y,s) |y+z|^\frac{\gamma}{p-1}\, dy \le C \int_{B(0,s^\frac{1}{\theta})} G(y,s) |y+z|^\frac{\gamma}{p-1} \, dy
\end{equation}
for almost all $\rho^\theta<s<T/3$.
First, we give an upper estimate to the integral on $B(0,s^{1/\theta})^c$.
By virtue of \eqref{eq:1.2},
we see that
\begin{equation}
\label{eq:4.3}
\begin{split}
\int_{B(0,s^\frac{1}{\theta})^c} G(y,s) |y+z|^\frac{\gamma}{p-1}\, dy
& = \int_{B(0,s^\frac{1}{\theta})^c} G(y,s) |y|^\frac{\gamma}{p-1} \left(\frac{|y+z|}{|y|}\right)^\frac{\gamma}{p-1}\, dy\\
& \le \int_{B(0,s^\frac{1}{\theta})^c} G(y,s) |y|^\frac{\gamma}{p-1} \left(1+\frac{|z|}{|y|}\right)^\frac{\gamma}{p-1}\, dy\\
& \le C\left(1+\frac{|z|}{s^\frac{1}{\theta}}\right)^\frac{\gamma}{p-1}\int_{B(0,s^\frac{1}{\theta})^c} G(y,s) |y|^\frac{\gamma}{p-1} \, dy\\
& \le C(|z|+s^\frac{1}{\theta})^\frac{\gamma}{p-1}
\end{split}
\end{equation}
for almost all $\rho^\theta<s<T/3$. 

Second, we give a lower estimate to the integral on $B(0,s^{1/\theta})$. 
Since $|z|>s^{1/\theta}>|y|$ for $y\in B(0,s^{1/\theta})$, we see that
\begin{equation}
\label{eq:4.4}
\begin{split}
\int_{B(0,s^\frac{1}{\theta})} G(y,s) |y+z|^\frac{\gamma}{p-1}\, dy
& \ge \int_{B(0,s^\frac{1}{\theta})} G(y,s) (|z|-|y|)^\frac{\gamma}{p-1}\, dy\\
& \ge (|z|-s^\frac{1}{\theta})^\frac{\gamma}{p-1} \int_{B(0,s^\frac{1}{\theta})} G(y,s) \, dy\\
& \ge C (|z|-s^\frac{1}{\theta})^\frac{\gamma}{p-1}
\end{split}
\end{equation}
for almost all $0<s<T/3$. 

Combining \eqref{eq:4.3} and \eqref{eq:4.4}, we obtain
\begin{equation}
\label{eq:4.5}
\begin{split}
&\int_{B(0,s^\frac{1}{\theta})^c} G(y,s) |y+z|^\frac{\gamma}{p-1}\, dy\le C\left(\frac{|z|+s^\frac{1}{\theta}}{|z|-s^\frac{1}{\theta}}\right)^\frac{\gamma}{p-1} \int_{B(0,s^\frac{1}{\theta})} G(y,s) |y+z|^\frac{\gamma}{p-1}\, dy
\end{split}
\end{equation}
for almost all $\rho^\theta<s<T/3$. 
Since $\rho^\theta<s<T/3$ and $|z|>T^{1/\theta}$, we have 
\begin{equation*}
%\label{eq:4.6}
\frac{|z|+s^\frac{1}{\theta}}{|z|-s^\frac{1}{\theta}} \le \frac{|z|+(T/3)^\frac{1}{\theta}}{|z|-(T/3)^\frac{1}{\theta}}
\le \frac{|z|+|z|/3^\frac{1}{\theta}}{|z|-|z|/3^\frac{1}{\theta}}
\le C.
\end{equation*}
This together with \eqref{eq:4.5} yields  \eqref{eq:4.2}. \\
\underline{\bf 2nd step.}
We give an upper estimate to the integral on $B(0,s^{1/\theta})$.
By \eqref{eq:2.1} and \eqref{eq:2.3}, we see that
\begin{equation}
\label{eq:4.7}
\begin{split}
\int_{B(0,s^\frac{1}{\theta})} G(y,s)|y+z|^\frac{\gamma}{p-1} \, dy
&\le C\int_{B(0,\rho)} G(y,\rho^\theta) \left|\frac{s^\frac{1}{\theta}}{\rho}y+z\right|^\frac{\gamma}{p-1}\,dy\\
&\le C\int_{B(0,\rho)} G(y,\rho^\theta) \left(\frac{s^\frac{1}{\theta}}{\rho}|y|+|z|\right)^\frac{\gamma}{p-1}\,dy\\
&\le C(|z|+s^\frac{1}{\theta})^\frac{\gamma}{p-1} \int_{B(0,\rho)} G(y,\rho^\theta) |y+z|^\frac{\gamma}{p-1}|y+z|^{-\frac{\gamma}{p-1}}\,dy\\
&\le C \left(\frac{|z|+s^\frac{1}{\theta}}{|z|-\rho} \right)^\frac{\gamma}{p-1} G(0,\rho^\theta) \int_{B(0,\rho)}  |y+z|^\frac{\gamma}{p-1}\,dy\\
&\le C \left(\frac{|z|+s^\frac{1}{\theta}}{|z|-\rho} \right)^\frac{\gamma}{p-1} \rho^{-N} \int_{B(z,\rho)}  |y|^\frac{\gamma}{p-1}\,dy\\
\end{split}
\end{equation}
for almost all $\rho^\theta<t<T/3$.
Since $\rho^\theta<s<T/3$ and $|z|>T^{1/\theta}$, we have
\[
\frac{|z|+s^\frac{1}{\theta}}{|z|-\rho} \le \frac{|z|+(T/3)^\frac{1}{\theta}}{|z|-(T/3)^\frac{1}{\theta}}
\le \frac{|z|+|z|/3^\frac{1}{\theta}}{|z|-|z|/3^\frac{1}{\theta}}
\le C.
\]
By combining \eqref{eq:4.2} and \eqref{eq:4.7}, we obtain \eqref{eq:4.1}.
\end{proof}
\begin{proof}[Proof of Theorem~\ref{Theorem:1.2}]
%{\bf Proof of Theorem~\ref{Theorem:1.2}.} 
We assume  $p=p_0$.
Let $\nu>0$ be a sufficiently small constant. Let $\rho$ be such that
\[
0<\rho<(\nu T)^\frac{1}{\theta}.
\]
Set $v(x,t) := u(x+z, t+(2\rho)^\theta)$ for almost all $x\in\RN$ and $t \in (0, T-(2\rho)^\theta)$. Since $u$ is a solution to \eqref{eq:1.1} in $\RN\times[0,T)$, it follows from \eqref{eq:1.3} that 
\begin{equation*}
%\label{eq:3.5}
v(x,t) = \int_{\RN} G(x-y,t)v(y,0) \, dy + \int_0^t \int_{\RN} G(x-y, t-s) |y+z|^{-\gamma} v(y,s)^p \, dyds
\end{equation*}
holds for almost all $x\in\RN$ and $t\in (0, T-(2\rho)^\theta)$. 
Then, we see that
\[
\int_{\RN} G(x,t) v(x,t) \, dy<\infty
\]
holds for almost all $0<t<[T-(2\rho)^\theta]/3$. Set
\[
\overline{w}(t):=\int_{\RN} G(x,t) v(x,t) \, dx
\]
for almost all $0<t<[T-(2\rho)^\theta]/3$. By an argument similar to that used in the proof of  \eqref{eq:1.6}, we obtain
\begin{equation}
\label{eq:4.8}
\begin{split}
\infty> \overline{w}(t) 
&\ge c_1 \mu(B(z,\rho)) t^{-\frac{N}{\theta}}\\
& + c_2 t^{-\frac{N}{\theta}} \int_{\rho^\theta}^t s^{\frac{N}{\theta}} \left(\int_{\RN} G(y,s) |y+z|^\frac{\gamma}{p-1}\,dy\right)^{-(p-1)} \overline{w}(s)^p \, ds
\end{split}
\end{equation}
for almost all $\rho^\theta<t<[T-(2\rho)^\theta]/3$.
By applying Lemma~\ref{Lemma:4.1} to \eqref{eq:4.8}, we have
\begin{equation*}
%\label{eq:4.8}
\begin{split}
\infty> \overline{w}(t) 
&\ge c_1 \mu(B(z,\rho)) t^{-\frac{N}{\theta}}\\
& + c_2 t^{-\frac{N}{\theta}} 
\left(\rho^{-N}\int_{B(z,\rho)}|y|^\frac{\gamma}{p-1}\,dy\right)^{-(p-1)}
\int_{\rho^\theta}^t s^{\frac{N}{\theta}}  \overline{w}(s)^p \, ds
\end{split}
\end{equation*}
for almost all $\rho^\theta<t<[T-(2\rho)^\theta]/3$.
Again, by an iteration scheme similar to the one used in  the proof of \eqref{eq:1.6}, we have 
\begin{equation*}
\begin{split}
\infty>\overline{w}(t)  
& \ge \left[\beta^p \left(\rho^{-N}\int_{B(z,\rho)} |y|^\frac{\gamma}{p-1} \, dy\right)^{-1}\mu(B(z,\rho)) \left(\log\frac{t}{\rho^\theta}\right)^\frac{1}{p-1} \right]^{p^k} \\
& \qquad\qquad\qquad\times t^{-\frac{N}{\theta}} \rho^{-N}\int_{B(z,\rho)} |y|^\frac{\gamma}{p-1} \, dy\left( \log \frac{t}{\rho^\theta}\right)^{-\frac{1}{p-1}}\\
&\ge \left[\beta^p \left(\rho^{-N}\int_{B(z,\rho)} |y|^\frac{\gamma}{p-1} \, dy\right)^{-1}\mu(B(z,\rho)) \left(\log\frac{T}{5\rho^\theta}\right)^\frac{1}{p-1} \right]^{p^k} \\
& \qquad\qquad\qquad\times t^{-\frac{N}{\theta}} \rho^{-N}\int_{B(z,\rho)} |y|^\frac{\gamma}{p-1} \, dy\left( \log \frac{t}{\rho^\theta}\right)^{-\frac{1}{p-1}}, \quad k=1,2,\cdots
\end{split}
\end{equation*}
for almost all $T/5<t<T/4$, where $\beta>0$ is a constant. Then it follows that
\[
\beta^p \left(\rho^{-N}\int_{B(z,\rho)} |y|^\frac{\gamma}{p-1} \, dy\right)^{-1}\mu(B(z,\rho)) \left[\log\left(e+\frac{T}{5\rho^\theta}\right)\right]^\frac{1}{p-1}\le 1
\]
for all $0<\rho<(\nu T)^{1/\theta}$.
Similarly to the proof of \eqref{eq:1.6}, we obtain 
\[
\left(\dashint_{B(z,\sigma)} |y|^\frac{\gamma}{p-1} \, dy\right)^{-1}\mu(B(z,\sigma)) \le C
\left[\log\left(e+\frac{T}{\sigma^\theta}\right)\right]^{-\frac{1}{p-1} }
\]
for all $z\in\RN$ with $|z|>T^{1/\theta}$ and $0<\sigma<T^{1/\theta}$. This is the desired estimate and the proof of Theorem~\ref{Theorem:1.2} is complete.
\end{proof}

%%%%%%%%%%%%%%%%%%%%%%%%%%%%%%%%%%%%
%%%%%%%%%%%%%%%%%%%%%%%%%%%%%%%%%%%%
\section{Proof of Theorem~\ref{Theorem:1.3}}
%%%%%%%%%%%%%%%%%%%%%%%%%%%%%%%%%%%%
%%%%%%%%%%%%%%%%%%%%%%%%%%%%%%%%%%%%
%
In this section, we prove Theorem~\ref{Theorem:1.3}. To simplify the notation, we set $V(x):=|x|^{-\gamma}$.
\begin{proof}[Proof of Theorem~\ref{Theorem:1.3}]
%{\bf Proof of Theorem~\ref{Theorem:1.3}.}
By Lemma~\ref{Lemma:2.2}, it is sufficient to construct a supersolution to \eqref{eq:1.1}.
Furthermore, it is sufficient to consider the case of $T=1$. Indeed, for any solution $u$ to \eqref{eq:1.1} in $\RN\times[0,T)$, where $0<T<\infty$, we see that $u_\lambda(x,t):=\lambda^{(\theta-\gamma)/(p-1)} u(\lambda x, \lambda^\theta t)$ with $\lambda=T^{1/\theta}$ is also a solution to \eqref{eq:1.1} in $\RN\times[0,1)$. 
Let $\alpha>1$ be sufficiently close to $N/\gamma$.
Set
\[
\rho(t):=t^{1-\frac{\gamma}{\theta}-\frac{\theta-\gamma}{\theta}\frac{1}{p-1}(p-\frac{1}{\alpha'})}
\]
and
\[
W(t):=S(t)\mu + \rho(t)(S(t)\mu^r)^\frac{1}{r\alpha'},
\]
where $\alpha'=\alpha/(\alpha-1)$.
We will show that $W(t)$ is a supersolution to \eqref{eq:1.1} in $\RN\times[0,1)$. 
Since $p>p_\gamma$, $\alpha>1$ is sufficiently close to $N/\gamma$ and \eqref{eq:1.8} holds,
 we see that
\begin{equation}
\label{eq:5.1}
1-\frac{\theta-\gamma}{\theta}\frac{1}{p-1}\left(p-\frac{1}{\alpha'}\right)>0
\end{equation}
and
\begin{equation}
\label{eq:5.2}
1+\frac{\theta-\gamma}{\theta}p\left(1-\frac{1}{p-1}\left(p-\frac{1}{\alpha'}\right)\right) -\frac{\theta-\gamma}{\theta}\frac{1}{\alpha'}  >0.
\end{equation}
Since $G$ satisfies \eqref{eq:2.4} and  \eqref{eq:2.5}, by  the H\"{o}lder inequality and the Jensen inequality, we have 
\begin{equation}
\label{eq:5.3}
\begin{split}
S(t-s)V(S(s)\mu)^p
&\le (S(t-s)V^\alpha)^\frac{1}{\alpha} (S(t-s)(S(s)\mu)^{p\alpha'})^\frac{1}{\alpha'}\\
&\le \|S(t-s)V^\alpha\|_{L^\infty(\RN)}^\frac{1}{\alpha} (\|S(s)\mu\|_{L^\infty(\RN)}^{p\alpha'-1} S(t)\mu)^\frac{1}{\alpha'}\\
&\le  \|S(t-s)V^\alpha\|_{L^\infty(\RN)}^\frac{1}{\alpha} \|S(s)\mu^r\|_{L^\infty(\RN)}^{\frac{1}{r}(p-\frac{1}{\alpha'})}(S(t)\mu^r)^\frac{1}{r\alpha'}.
\end{split}
\end{equation}
Since $1<\alpha<N/\gamma$, by \eqref{eq:1.9} and Lemma~\ref{Lemma:2.1} we have
\begin{equation}
\label{eq:5.4}
\|S(t-s)V^\alpha\|_{L^\infty(\RN)}^\frac{1}{\alpha} \le C(t-s)^{-\frac{\gamma}{\theta}}
\end{equation}
and
\begin{equation}
\label{eq:5.5}
\|S(s)\mu^r\|_{L^\infty(\RN)}^\frac{1}{r} \le CC_3 s^{-\frac{\theta-\gamma}{\theta}\frac{1}{p-1}}
\end{equation}
for almost all $0<s<t$.
Then by \eqref{eq:5.3}, \eqref{eq:5.4} and \eqref{eq:5.5} we have
\begin{equation*}
S(t-s)V(S(s)\mu)^p\le CC_3^{p-\frac{1}{\alpha'}} (t-s)^{-\frac{\gamma}{\theta}} s^{-\frac{\theta-\gamma}{\theta}\frac{1}{p-1}\left(p-\frac{1}{\alpha'}\right)}(S(t)\mu^r)^\frac{1}{r\alpha'}
\end{equation*}
for almost all $0<s<t$.
By \eqref{eq:5.1}, the right hand side is integrable with respect to $s$. Then we obtain
\begin{equation}
\label{eq:5.6}
\begin{split}
\int_0^t S(t-s)V(S(s)\mu)^p \, ds
&\le CC_3^{p-\frac{1}{\alpha'}} t^{1-\frac{\gamma}{\theta}-\frac{\theta-\gamma}{\theta}\frac{1}{p-1}\left(p-\frac{1}{\alpha'}\right)}(S(t)\mu^r)^\frac{1}{r\alpha'}\\
&=   CC_3^{p-\frac{1}{\alpha'}} \rho(t) (S(t)\mu^r)^\frac{1}{r\alpha'}
\end{split}
\end{equation}
for almost all $0<t<1$. Similarly to \eqref{eq:5.3}, we see that
\begin{equation*}
\begin{split}
S(t-s)V(\rho(s)(S(s)\mu^r)^\frac{1}{r\alpha'})^p
&\le \rho(s)^p (S(t-s)V^\alpha)^\frac{1}{\alpha}(S(t-s)(S(s)\mu^r)^\frac{p}{r})^\frac{1}{\alpha'}\\
&\le C\rho(s)^p (t-s)^{-\frac{\gamma}{\theta}}(S(t-s)(S(s)\mu^r)^p)^\frac{1}{r \alpha'}\\
&\le C\rho(s)^p (t-s)^{-\frac{\gamma}{\theta}}\|S(s)\mu^r\|_{L^\infty(\RN)}^\frac{p-1}{r\alpha'} (S(t)\mu^r)^\frac{1}{r\alpha'}\\
&\le CC_3^{\frac{p-1}{\alpha'}}(t-s)^{-\frac{\gamma}{\theta}} s^{\frac{\theta-\gamma}{\theta}p\left(1-\frac{1}{p-1}\left(p-\frac{1}{\alpha'}\right)\right) -\frac{\theta-\gamma}{\theta}\frac{1}{\alpha'} }(S(t)\mu^r)^\frac{1}{r\alpha'}\\
\end{split}
\end{equation*}
for almost all $0<s<t$. By \eqref{eq:5.2} we have
\begin{equation}
\label{eq:5.7}
\begin{split}
&\int_0^t S(t-s)V(\rho(s)(S(s)\mu^r)^\frac{1}{r\alpha'})^p \, ds\\
&\le CC_3^{\frac{p-1}{\alpha'}} t^{1-\frac{\gamma}{\theta}+\frac{\theta-\gamma}{\theta}p\left(1-\frac{1}{p-1}\left(p-\frac{1}{\alpha'}\right)\right) -\frac{\theta-\gamma}{\theta}\frac{1}{\alpha'} }(S(t)\mu^r)^\frac{1}{r\alpha'}\\
&= CC_3^{\frac{p-1}{\alpha'}} \rho(t) (S(t)\mu^r)^\frac{1}{r\alpha'}\\
\end{split}
\end{equation}
for almost all $0<t<1$. Combining \eqref{eq:5.6} and \eqref{eq:5.7}, we see that
\begin{equation*}
\begin{split}
&S(t)\mu + \int_0^t S(t-s)V(W(s))^p \, ds\\
&\le S(t)\mu + 2^{p-1}\int_0^t S(t-s)V(S(s)\mu)^p \, ds + 2^{p-1} \int_0^t S(t-s)V(\rho(s)(S(s)\mu^r)^\frac{1}{r\alpha'})^p \, ds\\
&\le S(t)\mu + C(C_3^{p-\frac{1}{\alpha'}}+C_3^\frac{p-1}{\alpha'})\rho(t) (S(t)\mu^r)^\frac{1}{r\alpha'}\\
\end{split}
\end{equation*}
for almost all $0<t<1$. Taking a sufficiently small constant $C_3>0$ if necessary,
$W(t)$ is a supersolution to \eqref{eq:1.1} in $\RN\times[0,1)$.
Thus, the proof of Theorem~\ref{Theorem:1.3} is complete.
\end{proof}

\vspace{5pt}
\noindent
{\bf Acknowledgments.} 
The first author of this paper is grateful to Professor K. Ishige for mathematical discussions and proofreading of the manuscript. 
%Finally,
%the authors of this paper would like to thank the referees for carefully reading the manuscript and relevant remarks. 

%%%%%%%%%%%%%%%%%%%%%%%%%%%%%%%%%%%%%%
%%%%%%%%%%%%    references    %%%%%%%%%%%%%%%%%%
%%%%%%%%%%%%%%%%%%%%%%%%%%%%%%%%%%%%%%
  
%%%%%%%%%%%%%%%%%%%%%%%%%%%%%%%%%%%%
%%%%%%%%%%%%%%%%%%%%%%%%%%%%%%%%%%%% 

\end{document}